\documentclass[12pt]{amsart}
\usepackage[latin1]{inputenc}
\usepackage{amssymb}
\usepackage{dsfont}
\usepackage{tikz}
\usepackage[ruled,vlined,norelsize]{algorithm2e}
\usepackage{algpseudocode}
\usepackage{mathptmx}
\usepackage{amscd}
\usepackage{cancel}
\usetikzlibrary{arrows}
\usepackage{enumitem}

\numberwithin{equation}{section}

\def\ZZ{{\mathds Z}}

\def\NN{{\mathds N}}
\newcommand{\KK}{\mathbb{K}}

\def\sdepth{\operatorname{sdepth}}

\def\size{\operatorname{size}}

\def\min{\operatorname{min}}

\def\Ass{\operatorname{Ass}}
\def\height{\operatorname{height}}

\def\frk{\frak}

\DeclareMathOperator{\lcm}{lcm}

\newcommand*\xbar[1]{%
  \hbox{%
    \vbox{%
      \hrule height 0.5pt 
      \kern0.5ex
      \hbox{%
        \kern-0.1em
        \ensuremath{#1}%
        \kern-0.1em
      }%
    }%
  }%
}
\newcommand{\x}[2]{X_{{#1},{#2}}}           

\let\frak=\mathfrak
\let\Bbb=\mathbb
\let\phi=\varphi
\def\NZQ{\Bbb}
\def\NN{{\NZQ N}}

\def\ZZ{{\NZQ Z}}

\newtheorem{lemma}{Lemma}[section]
\newtheorem{corollary}[lemma]{Corollary}
\newtheorem{theorem}[lemma]{Theorem}
\newtheorem{proposition}[lemma]{Proposition}

\theoremstyle{definition}
\newtheorem{definition}[lemma]{Definition}
\newtheorem{remark}[lemma]{Remark}
\newtheorem{example}[lemma]{Example}

\usepackage{color}

\newcommand{\set}[1]{\{#1\}}
\newcommand{\with}{\,:\,}

\newcommand{\gen}[1]{{#1}_\varepsilon}
\newcommand{\MM}{\mathcal{M}}
\newcommand{\ve}{\varepsilon}
\newcommand{\xf}{\mathbf{x}}

\begin{document}

\title{On the behavior of the size of a monomial ideal}

\author{Bogdan Ichim and Andrei Zarojanu}

\address{Simion Stoilow Institute of Mathematics of the Romanian Academy, Research Unit 5, P.O. Box 1-764, 014700 Bucharest, Romania}
\email{bogdan.ichim@imar.ro}

\address{Simion Stoilow Institute of Mathematics of the Romanian Academy, Research Unit 5, P.O. Box 1-764, 014700 Bucharest, Romania}
\email{andrei\_zarojanu@yahoo.com}


\begin{abstract}
In this paper we study the behavior of the $\size$ of a monomial ideal under polarization and under generic deformations.
As an application, we extend a result relating the $\size$ and the Stanley depth of a squarefree monomial ideal obtained by Herzog, Popescu and Vladoiu, together with a parallel result obtained by Tang.
\end{abstract}

\maketitle

\section{Introduction}

Let $S=\KK[X_1,...,X_n]$, with $\KK$ a field, and let $I\subset S$ be a monomial ideal.
The notion of \emph{size} of a monomial ideal was introduced by Lyubeznik in \cite{L}.
In time, it has been used by several authors, see for example  \cite{HPV}, \cite{HTT}, \cite{P}, \cite{S} and \cite{T}.

Several algebraic or combinatorial invariants associated to a monomial ideal are known to have a nice behavior under \emph{polarization}. For example, see \cite{F}, \cite{GPW} or \cite{IKM}.
In the first part of this paper we study the behavior of the $\size I$ under polarization.  In Section \ref{sec:polarization} we establish that
$$
\size I^p \leq \size I + c,
$$
where $I^p\subset S'=\KK[X_1,...,X_{n'}]$ is the polarization of $I$ and $c=n'-n$ (see Theorem \ref{th:first}). The equality does not hold in general, as shown in Example \ref{ex:no_eq}.

In the main result of this paper, that is
Theorem \ref{th:main}, we provide a complete description of the (particular) situation when the equality
$$
\size I^p = \size I + c
$$
does hold.

A counterexample by H. Shen shows that the second statement of \cite[Lemma 3.2]{HPV} is false when $I$ is not
squarefree. It follow that the proof \cite[Theorem 3.1]{HPV} is correct only when I is squarefree, and that the statement of \cite[Theorem 3.1]{HPV} is in fact a conjecture in general.
As an application of out main result we deduce in Corollary \ref{cor:sdepth} that this conjecture is true under the conditions described in Theorem \ref{th:main}.
In the same Corollary and under the same conditions we also obtain an extension of \cite[Theorem 3.2]{T}.

The notion of \emph{deformation} of a monomial ideal was introduced by Bayer et al.~\cite{BPS} and further developed in Miller et al.~\cite{MSY}. The most important deformations are the generic deformations, which attracted the attention of several researchers, see for example \cite{A} or \cite{IKM2}. In the last part of this paper (Section \ref{sec:deform}) we briefly study the behavior of the $\size I$ under generic deformations. We find that
$$
\size I_{\epsilon}\leq \size I,
$$
where $I_{\epsilon}$ is a generic deformation of $I$ (see Proposition \ref{prop:def}).

\section{Prerequisites}\label{Pre}

  Let $S=\KK[X_1,...,X_n]$, with $\KK$ a field. For $n \in \NN$ we use the notation $[n]:=\{1,\ldots ,n\}$.

\subsection{The size of an ideal} \label{subsec:size} In this Subsection we recall the definition of $\size$ and we make some easy remarks;
these will be needed in the sequel.

Let $I\subset S$  be a monomial ideal and $I=\bigcap\limits_{i=1}^sU_i$ an irredundant primary decomposition
of $I$, with $U_i$ monomial ideals. Let $U_i$ be $V_i$-primary. Then each $V_i$ is a monomial prime ideal and $\Ass(S/I)=\{V_1,\ldots,V_s\}$.

\begin{definition}
Following  Lyubeznik \cite[Proposition 2]{L} the {\em size} of $I$, denoted in the following by $\size I$,
is the number $v + (n - h) - 1$,  where $v$ is  the minimum number $t \le s$ such that there exist $i_1<\cdots< i_t$  with
\[
\sqrt{\sum_{j=1}^tU_{i_j}}=	\sqrt{\sum_{i=1}^sU_i},
\]
and where $h=\height \sum_{i=1}^sU_i$.
\end{definition}

A monomial ideal is \emph{irreducible} if it is generated by powers of some variables.
An irreducible decomposition of a monomial ideal $I$ is an expression $I = I_1 \cap \ldots \cap I_r$ with the $I_j$ irreducible. Every irreducible monomial ideal is primary, so an irreducible
decomposition is a primary decomposition \cite[Theorem 1.3.1]{HH}.

\begin{remark} Observe that $\sqrt{\sum_{j=1}^tU_{i_j}} =\sum_{j=1}^tV_{i_j}$ and $\sqrt{\sum_{i=1}^sU_i}=\sum_{i=1}^sV_i$, so that the size of $I$ only depends on the set of associated prime ideals of $S/I$.
Indeed, consider an irredundant irreducible decomposition
$$I=\bigcap\limits_{l=1}^r Q_l,\ r \leq s,$$
where $Q_l$ are monomial ideals and where $P_l = \sqrt{Q_l}$.
Then, for each $i \in [s]$ we have that $U_i = \bigcap\limits_{j=1}^{k_i} Q_{l_j}$ where $P_{l_j} = V_i$. Using the definition of $\size I$  we get the same result as above.  From now on we will only consider irredundant irreducible decompositions.
\end{remark}

\begin{remark}
Working with $\size$ we can assume that $\sum_{i=1}^{r} P_{i} = \frk{m}$, where $\frk{m}$ is the maximal ideal. Otherwise, set $X= \{ X_1,\ldots,X_n \}$, $Z= \{X_k|\ X_k \notin \sum_{i=1}^{s} P_{i} \}$, $T=\KK[X \setminus Z]$ and let $J=I \cap T$. Then the sum of the associated prime ideals of $J$ is the maximal ideal of $T$ and
\begin{center}
 $\size I = \size J + |Z|$.
\end{center}
\end{remark}

\begin{remark}
Let $I \subset S$ be a monomial ideal and let $I=\bigcap\limits_{i=1}^rQ_i$ be an irredundant decomposition
of $I$ as an intersection of irreducible ideals, where $\sqrt{Q_i} = P_i.$ Then $\sqrt{I} = \bigcap\limits_{i=1}^rP_i$ and so we have that $\size I = \size \sqrt{I}.$
In general it easy to see that, if $\Ass S/I \subseteq \Ass S/J$ for two monomial ideals, then $\size I \geq \size J$.
\medskip

\end{remark}

\subsection{The polarization of a monomial ideal} \label{subsec:polarization}
In the following we study the behavior of the size of a monomial ideal under polarization.
We recall the definition of polarization following Herzog and Hibi \cite{HH}.
Let $I \subset S$ be a monomial ideal with generators $u_1, \ldots, u_m$, where $u_i=\prod_{j=1}^{n} X_j^{a_{ij}}$ for $i=1, \ldots , m$. For each $j$ let $a_j=\max \set{a_{ij} : i=1, \ldots , m}$.
Set $a=(a_1,\ldots,a_n)$ and $S'$ to be the polynomial ring
\[
S':=\KK[X_{k,l} \with 1 \leq k \leq n,\ 1 \leq l \leq a_j].
\]

Then the \emph{polarization of $I$} is the squarefree monomial ideal $I^p \subset S'$ generated by $v_1, \ldots , v_m$, where
\[
v_i=\prod_{k=1}^{n} \prod_{l=1}^{a_{ij}} X_{k,l} \ \ \ \mbox{for} \ \ i=1, \ldots, m.
\]
\medskip

The \emph{Stanley depth} of an $S$-module $M$ is a combinatorial invariant denoted in the following by $\sdepth M$.
We skip the details since this invariant will only appear briefly in Corollary \ref{cor:sdepth}.
For an excellent account on the subject, the reader is referred to Herzog's survey \cite{H}. The following Theorem follows immediately from the main result of \cite{IKM}.

\begin{theorem}\label{teo:sdepth} Let $I \subset S=\KK[X_1,...,X_n]$ be a monomial ideal and $I^p\subset S'=\KK[X_1,...,X_{n'}]$ be the polarization of $I$. Set $c=n'-n$. Then
\begin{enumerate}
\item $\sdepth I^p = \sdepth I + c;$
\item $\sdepth S^p/I^p = \sdepth S/I + c.$
\end{enumerate}
\end{theorem}

Finally, we recall the most important known results relating $\sdepth$ and $\size$.

\begin{theorem}\label{teo:sdepth_size} Let I be as squarefree monomial ideal of $S$.Then
\begin{enumerate}
\item $\sdepth I \ge \size I + 1$ (see \cite[Theorem 3.1]{HPV});
\item $\sdepth S/I \ge \size I $ (see \cite[Theorem 3.2]{T}).
\end{enumerate}
\end{theorem}

For an extension of Theorem \ref{teo:sdepth_size} see \cite{Fak}.


\section{The behavior of size under polarization} \label{sec:polarization}

Let $S=\KK[X_1,...,X_n]$, with $\KK$ a field, and $I \subset S$ be a monomial ideal. Throughout the section, we fix  $I=\bigcap\limits_{i=1}^rQ_i$ to be the unique irredundant irreducible
decomposition of $I$. For each $i\in [r]$ we have
$$Q_i=(X_1^{a_1^i},\ldots,X_n^{a_n^i}),$$
where $a_k^i \in \mathbb{N}$ for all $k \in [n]$. In this writing we use the following convention: Assume $Q_i$ to be $P_i$-primary. Then, if $X_k \not\in P_i$, we set $a_k^i = 0$ and $X_k^{a_k^i}=0$.

Let $a_k = \max \{a_k^1,\ldots,a_k^r\}$, for all $k\in [n]$. Denote by $n'=\sum\limits_{k=1}^n a_k$ and set $c = n'-n$. Set

\[
S':=\KK[X_{k,l} \with 1 \leq k \leq n, 1 \leq l \leq a_j].
\]

Let $I^p \subset S'$ be the polarization of $I$ and $Q_i^p \subset S'$ be the polarization of $Q_i$. Then
$$I^p=\bigcap\limits_{i=1}^rQ_i^p$$ by \cite[Proposition 2.3]{F}. Moreover, by \cite[Proposition 2.5]{F}, it holds
\[ Q_i^p=\bigcap_{\stackrel{\mbox{$\scriptstyle 1 \leq b_k \leq
a_k^i$}}{ \mbox{$\scriptstyle 1 \leq k \leq n$}}} (\x{1}{b_1}, \ldots,
\x{n}{b_n})\]
with the convention that, if $a_k^i = 0$, then $b_k = X_{k,0} = 0$.

\begin{definition}

 Remark that for all $i \in [r]$ there exists at least one index $k \in [n]$ such that $a^i_k \geq a^j_k$ for all $j \in [r]$. We say that $a^i_k$ is a \textit{top power} for $Q_i$. For a subset $N \subseteq [n]$ we  set $B^i_N = \{ a^i_k : a^i_k$ is a top power for $Q_i$ and $k \in N \}$. Note that $B^i_{[n]} \neq \emptyset$ for all  $i \in [r]$.   For a subset $R \subseteq [r]$ we introduce the notion \textit{top base} (denoted by $C$ by the following recursive algorithm). Let $M \in \mathcal{M}_{r,n} (\ZZ)$ be the matrix with $M_{ij} = a^i_j$, for all $i \in [r],\ j \in [n]$.
\end{definition}

\allowdisplaybreaks
\begin{algorithm}[H]\label{algo:topbase}
\SetKwFunction{ReadTopPowers}{{\bf ReadTopPowers}}
\SetKwFunction{AddElement}{{\bf AddElement}}
\SetKwFunction{Beg}{begin}
\SetKwFunction{En}{end}
\SetKwFunction{min}{min}
\SetKwFunction{max}{max}
\SetKwFunction{BuildTopBase}{{\bf BuildTopBase}}
\SetKwFunction{NewEmptyList}{{\bf NewEmptyList}}
\SetKwData{Int}{Int}
\SetKwData{Container}{Container}
\SetKwData{List}{List}
\SetKwData{Vector}{Vector}
\SetKwData{Set}{Set}
\SetKwData{Sets}{Sets}
\SetKwData{Matrix}{Matrix}
\caption{Function which computes a top base}

\KwData{$r,n \in \NN$, a \Matrix $M\in \mathcal{M}_{r,n}(\ZZ)$ and the \Sets $R\subseteq [r],\ N\subseteq [n]$}
\KwResult{A List $C$ containing a top base}
\nl \Vector $C$ = NewVector(r,0), R = [r], N = [n]\;
\nl \List \BuildTopBase{M,R,N}\;
\Begin{
\nl \If {$R = \emptyset$}{
\nl \Return{$C$}\;}
\nl $i = \min{R}$\;
\nl \List $B_N^i$ = \ReadTopPowers(i,N,M)\;
\nl \eIf {$B_N^i = \emptyset$}{
\nl \Set $\tilde{R} = R \setminus \{i\}$\;
\Return{\BuildTopBase{M,$\tilde{R}$,N}\;}
}{
\nl \For { j=\Beg{N} \KwTo j=\En{N} }{
\nl \If {$a_j^i$ = \max{$B_N^i$}}{
\nl $C[i] \longleftarrow (a_j^i,i,j)$)\;
\nl \Set $\tilde{R} = R \setminus \{i\}$\;
\nl \Set $\tilde{N} = N \setminus \{j\}$\;
\nl \Return{\BuildTopBase{M,$\tilde{R},\tilde{N}$}}\;
}
}
}
}
\end{algorithm}

Below we describe the key steps.

\begin{itemize}
	\item line 1. We initialize the Vector $C$ to have $0$ on all components and lenght $r$.
	\item line 6. We read the top powers from $M$ on the line $i$. If we find any of these on the columns $j \in N$ then we include them in the list $B^i_N.$
	\item line 7. If there aren't any top powers on line $i$ and the columns $j \in N$ from $M$ then we skip to the next line in $M.$
\end{itemize}

Remark that a top base is not unique and it depends on the choice of the maximal top powers from each line, as the following example shows.

\begin{example}
Let $I= (x^{10},y^{10},z) \cap (x^{10},y^2) \cap (x,z^4).$ Then we have that

\[M = \begin{pmatrix}
  \textbf{10} & \textbf{10} & 1 \\
  \textbf{10} & 2 & 0 \\
  1 & 0 & \textbf{4}
    \end{pmatrix}\]
where the top powers are boldfaced.

We see that $\{a_1,a_2,a_3\} = \{10,10,4\}$ and we start the algorithm above to compute a top base.
At the first step we can select $c_1 = a_1^1 = 10.$

\[ \begin{pmatrix}
   \textbf{10} & \textbf{10} & 1 \\
  \textbf{10} & 2 & 0 \\
  1 & 0 & \textbf{4}
    \end{pmatrix}
    \longrightarrow
    \begin{pmatrix}
      \cancel{\textbf{10}} & \cancel{\textbf{10}} & \cancel{1} \\
  \cancel{\textbf{10}} & 2 & 0 \\
  \cancel{1} & 0 & \textbf{4}
    \end{pmatrix}\]

Now in the second line we don't have any top powers, thus $c_2=0$ so we go to the third line from where we get $c_3 = 4.$
So we obtained the top base $\{10,0,4\}.$

    \[    \begin{pmatrix}
      \cancel{\textbf{10}} & \cancel{\textbf{10}} & \cancel{1} \\
  \cancel{\textbf{10}} & 2 & 0 \\
  \cancel{1} & 0 & \textbf{4}
    \end{pmatrix}
    \longrightarrow
    \begin{pmatrix}
      \cancel{\textbf{10}} & \cancel{\textbf{10}} & \cancel{1} \\
  \cancel{\textbf{10}} & \cancel{2} & \cancel{0} \\
  \cancel{1} & 0 & \textbf{4}
    \end{pmatrix}\]

Now, if we choose $c_1 = a_2^1$ we obtain a different top base $\{10,10,4\}.$

  \[   \begin{pmatrix}
   \textbf{10} & \textbf{10} & 1 \\
  \textbf{10} & 2 & 0 \\
  1 & 0 & \textbf{4}
    \end{pmatrix}
    \longrightarrow
     \begin{pmatrix}
      \cancel{\textbf{10}} & \cancel{\textbf{10}} & \cancel{1} \\
  \textbf{10} & \cancel{2} & 0 \\
  1 & \cancel{0} & \textbf{4}
    \end{pmatrix}
    \longrightarrow
    \begin{pmatrix}
      \cancel{\textbf{10}} & \cancel{\textbf{10}} & \cancel{1} \\
  \cancel{\textbf{10}} & \cancel{2} & \cancel{0} \\
  \cancel{1} & \cancel{0} & \textbf{4}
    \end{pmatrix}\]

\end{example}

\begin{definition} \label{defbar}
Consider the top base $C$ from the above algorithm and set $c_i = C[i][1], i\in [r]$. In the following we denote
$$
\xbar{Q^p_{i,j}}=(\x{1}{\min\{j,a_1^i\}},\ldots,\x{n}{\min\{j,a_n^i\}}) \text{ for } j\in [c_i],
$$
where $X_{i,0} = 0$ and if $c_i = 0$, then consider $\xbar{Q_{i,j}} = (1)$ in intersections of ideals and $\xbar{Q_{i,j}} = (0)$ in sums of ideals,
$$
\xbar{Q^p_i}= \bigcap\limits_{j=1}^{c_i} \xbar{Q^p_{i,j}} = (\x{1}{1},\ldots,\x{n}{1}) \cap (\x{1}{2},\ldots,\x{n}{2}) \cap \ldots \cap (X_{1,\min\{c_i,a_1^i\}},\ldots,X_{n,\min\{c_i,a_n^i\}})
$$
and
$$
\xbar{I^p}=\bigcap\limits_{i=1}^{r}\xbar{Q^p_i}.
$$
\end{definition}

\begin{remark}
Following Algorithm 1 we have that $\{c_1,\ldots,c_r\} \setminus \{0\} \subset \{a_1,\ldots,a_n\}.$ Thus $\sum\limits_{i=1}^r c_i \leq \sum\limits_{k=1}^n a_k$, which implies that $\sum\limits_{i=1}^r c_i \leq c+r.$
\end{remark}

\begin{definition} Set $A = \big\{ \{ j_1, \ldots, j_{w+1} \} \mbox{ such that } \sqrt{\sum_{k=1}^{w+1}Q_{j_k}}=\sqrt{\sum_{j=1}^rQ_j} \},\ w = \size I \big\}.$
We also use the notations $|A| = m$ and $A = \cup_{h=1}^m A_h.$
\end{definition}

\begin{remark} \label{size}
We may suppose (eventually after renumbering the ideals $Q_i$) that we always have $A_1 = \{1,2,\ldots,w + 1\} \in A,$ where $w = \size I.$
\end{remark}

\begin{theorem}\label{th:first}  $\size I^p \leq \size I + c.$
\end{theorem}

\begin{proof}
From the definition of $\size$ we see that $\size Q_i^p = \size \xbar{Q_i^p}$ and that
 $\size I^p = \size \xbar{I^p}$, where
 $$
 \xbar{I^p} = \bigcap\limits_{i=1}^r \xbar{Q_i^p}= \bigcap\limits_{\substack{1 \leq i \leq r \\ 1 \leq j \leq c_i}} \xbar{{Q_i^p}_j},
 $$
therefore $\xbar{I^p}$ has an irredundant decomposition composed of $D=\sum\limits_{i=1}^r c_i$ irreducible monomial ideals.

According to Remark \ref{size} we can assume that $\sqrt{Q_j} \subset \sqrt{\sum\limits_{i=1}^{w+1} Q_i}$, for all $t+2 \leq j \leq r$ and set $a=r-(w+1)$. We see that $\sqrt{Q_j} \subset \sqrt{\sum\limits_{i=1}^{w+1} Q_i}$,
for all $w+2 \leq j \leq r$ implies that $\xbar{{Q_j^p}_1} \subset \sum\limits_{i=1}^{w+1} \xbar{{Q_i^p}_1}$, for all $w+2 \leq j \leq r$.

Then
\begin{equation} \label{eq1}
\size I^p = \size \xbar{I^p} = \size \Bigg[\Bigg( \bigcap\limits_{\substack{1 \leq i \leq w+1 \\ 1 \leq j \leq c_i}} \xbar{{Q_i^p}_j}\Bigg) \cap  \Bigg(\bigcap\limits_{\substack{w+2 \leq i \leq r \\ 2 \leq j \leq c_i}} \xbar{{Q_i^p}_j}\Bigg)\Bigg].
\end{equation}

Notice that the last term has $D-a$ intervals.
Then $\size I^p \leq D-a-1 = D - r +w + 1 - 1 \leq c + w = \size I + c.$
\end{proof}

\begin{remark} \label{missingindex}
We see from equation \ref{eq1} that $\size I^p \leq \size I + \sum\limits_{\substack{i=1 \\ c_i > 0}} ^r (c_i - 1) \leq \size I + c.$ Thus, if $\{a_i : a_i > 1, i \in [n]\} \nsubseteq \{c_1,\ldots, c_r\}$, we have that
$$\size I^p \leq \size I + \sum\limits_{\substack{i=1 \\ c_i > 0}}^r (c_i - 1) < \size I + \sum\limits_{i=1}^n (a_i - 1) = \size I + c.$$
\end{remark}

\begin{example}\label{ex:no_eq}
Let $I=(X_1^2,X_2^2) \cap (X_3^2,X_4^2) = Q_1 \cap Q_2 \subset \KK[X_1,X_2,X_3,X_4]$ a monomial ideal. Then $\size I^p = 3 < \size I + c = 5.$
Indeed we see that $\size I = 1$ and that $c=4.$ Also note that
$$\size I = \size \xbar{{Q_1^p}_1} \cap \xbar{{Q_1^p}_2} \cap \xbar{{Q_2^p}_1} \cap \xbar{{Q_2^p}_2} = 3.$$
\end{example}


Let $I \subset S=\KK[X_1,\ldots,X_n]$ be a monomial ideal and let $I=\bigcap\limits_{i=1}^rQ_i$ an irredundant decomposition
of $I$ as an intersection of irreducible ideals, where the $Q_i=(X_1^{a_1^i},\ldots,X_n^{a_n^i}), a_j^i \in \mathbb{N}, j \in [n]$ are monomial ideals. Let $Q_i$ be $P_i$-primary.
For every $k \in [n]$, we  set:

$ T_k = \{ t : X_k ^ s \in G(Q_t), s \geq 2 \}$;

$ L_k = \{ l : X_k \in P_l \} $;

$ U_k = \{ u : X_k \in Q_u, \max(B^u_{[n]}) > 1 \}$.

We may now formulate the main result of this paper, which fully describes the (particular) cases when the equality holds.

\begin{theorem} \label{th:main}Let $I \subset S$ be a monomial ideal. Then  $\size I^p = \size I + c$ if and only if $Q_i = < \{X_{i_1}^{a_{i_1}^i},X_{i_2},\ldots X_{i_s} \}|\ s \in [n],\ a_{i_1} \geq 1 >$ for all $i \in [r]$ and,
if $X_k \in P_i \cap P_j,\ 1 \leq i < j \leq n,\ k \in [n]$, then one of the following is true:
\begin{enumerate}
	\item $T_k = \emptyset$ or
	\item $T_k \neq \emptyset$ and
	
	$(A)$  $\forall t \in T_k,\ \{t,l\} \nsubseteq A_h, \forall\ l \in L_k \setminus \{t\}, \forall\ h \in [m]$ and
	
	$(B)$ if there exists $t \in T_k \cap A_h$ for some $h \in [m]$, then $U_k = \emptyset.$
\end{enumerate}
\end{theorem}

\begin{proof}

"$\Longleftarrow$"

$(1)$ Let $Q_i = (X_i^{a_i^i},X_{i_1},\ldots,X_{i_t}),\ i\in [r],\ a_i > 1,\ t \in [n]$. If such a $Q_i$ does not exist, then $c=0$ and $I=I^p$. Thus we have that
$X_i \notin P_j$ for all $j \in [r] \setminus \{i\}$, thus $i \in A_h$, for all $h \in [m]$. We may assume that $Q_i,\ i \in [s]$ have a top power $\geq 2$ and that $Q_i,\ i > s$ are generated by variables, that is $\max(B^i_{[n]})=1$, for all $i > s$.  Then the ideals in the second intersection from the last term in the equation \ref{eq1} are generated by variables. Then we get that
$$\size I^p = \sum_{i=1}^s c_i + w - s = \sum_{i=1}^n (a_i - 1) + w = \size I + c,$$
because for all $j \in [a_i^i]$ and  $i\in [s]$ we only have the ideal $( \x{i}{j},\x{i_1}{1},\ldots,\x{i_t}{1})$ to cover $\x{i}{j}$.

$(2)$ Let $\size I = w$. Consider that $Q_{w+2},\ldots,Q_s$ have top powers $\geq 2$ and that $Q_{s+1},\ldots,Q_r$ have top powers $=1$. Then according to algorithm \ref{algo:topbase} we have that $c_i \leq 1$, for all $i > s$ and using equation \ref{eq1} we see that for computing $\size I^p$ we don't need the ideals $Q_i,\ i>s.$ Moreover, we shall consider only the ideals $Q_i,\ w+1 < i \leq s$ with $c_i > 1$. Thus we may suppose that $c_i > 1,\ w+1 < i \leq s.$

If, for example, we have that $C[i] = (a_i^j,i,i),\ i \geq w+2$, we show that in the sum of the other intervals we can cover only one variable from $\{\x{i}{1},\ldots,\x{i}{a_i^j}\}$ and that variable is $\x{i}{1}$ from our choice in definition \ref{defbar}. Indeed, if we have $x_i^t \in G(Q_1),\ t<a_i^j$, then $c_1$ can be at most $1$, so we cover $\x{i}{1}$. Condition $(A)$ tells us that $x_i \notin P_j,\ 1 < j \leq w+1.$ Condition $(B)$ tells us that if we have, for example, $C[1] = (a_1^d,1,1)$ and $C[i]=(a_i^j,i,i),\ i \neq 1,\ a_i^j > 1$, then $x_1 \notin G(Q_i)$, so that we can not cover $\x{1}{1}$ in the sum of the other intervals, that is $\xbar{{Q_i^p}_l},\ 2 \leq l \leq a_i^j$.

It follows that
$$\size I^p = \sum_{i=1}^{w+1} c_i + \sum_{i=w+2}^s (c_i-1) - 1 = \size I + \sum_{i=1}^{w+1} (c_i-1) + \sum_{i=w+2}^s (c_i-1) = \size I + c.$$
\bigskip

"$\Longrightarrow$"

Assume that $\size I^p = \size I + c$. Then Remark \ref{missingindex} gives us the following inclusion $\{a_i : a_i > 1,\ i \in [n]\} \subset \{c_1,\ldots, c_r\}$. Let $\size I = w$.

First suppose that $X_1^{a_1^1},X_2^{a_2^1} \in G(Q_1),\ a_1^j > 1,\ j \in [2]$.  If $c_1 > a_1^1$, then using Remark \ref{missingindex} we get that there exists $1 < j \leq r$ such that $c_j = a_1^j = a_1$. Then from equation \ref{eq1} we can skip the ideal $\xbar{{Q_j^p}_2}$ because we have $\x{1}{2} \in \xbar{{Q_1^p}_2}$ and we find all the other variables $\x{p}{2},\ 2 \leq p \leq n$ in $\xbar{{Q_w^p}_2}$, where $c_w = a_p$. We will call this procedure \textit{the absence of variable $\x{1}{2}$}. Thus we get that $\size I^p < \size I + c$.

Now, if $c_1 = a_1^1=a_1$ let $2 \leq j \leq r$ such that $c_j = a_2$. Then again, using \textit{the absence of variable $\x{2}{2}$}, we get that $\size I^p < \size I + c$.
Thus we see that $a_i^j = 1$ except for at most one $i \in [n]$, for all $j \in [r]$.


Suppose that $\{1,2\} \subset A_1,\ X_1^{a_1^1} \in G(Q_1),\ X_1^{a_1^2} \in G(Q_2),\ a_1^1 > 1,\ a_1^2 \geq 1,\ a_1^1 \geq a_1^2.$  If $c_1=a_1^1$,  then
\begin{align*}
\size I^p & \leq \size   \Bigg[\Bigg( \bigcap\limits_{\substack{1 \leq i \leq w+1 \\ 1 \leq j \leq c_i}} \xbar{{Q_i^p}_j}\Bigg) \cap  \Bigg(\bigcap\limits_{\substack{w+2 \leq i \leq r \\ 2 \leq j \leq c_i}}
             \xbar{{Q_i^p}_j}\Bigg)\Bigg]\\
          & \leq \size \Bigg[\Bigg(\bigcap\limits_{\substack{2 \leq i \leq w+1 \\ 1 \leq j \leq c_i}} \xbar{{Q_i^p}_j}\Bigg) \cap  \Bigg(\bigcap\limits_{\substack{w+2 \leq i \leq r \\ 2 \leq j \leq c_i}} \xbar{{Q_i^p}_j} \cap \xbar{{Q_1^p}_2} \ldots \cap \xbar{{Q_1^p}_{a_1^1}}\Bigg)\Bigg]\\
          & \leq \sum\limits_{i=2}^{w+1} c_i + \sum\limits_{i=w+2}^r (c_i-1) + c_1 - 1 \\
          & =  \size I - 1 + c.
\end{align*}
We skipped the ideal $\xbar{{Q_1^p}_1}$ because $\xbar{{Q_1^p}_1} \subset \xbar{{Q_1^p}_2} + \sum_{j=2}^{w+1} \xbar{{Q_j^p}_1}$.

If $c_1 > a_1^1$, then there exists $2 \leq j \leq r$ such that $c_j = a_1^j = a_1 > 1$ and in equation \ref{eq1} we can skip over the ideal $\xbar{{Q_j^p}_1}$ because $\xbar{{Q_j^p}_1} \subset \xbar{{Q_j^p}_2} + \sum\limits_{\substack{i=1 \\ i \neq j}}^{w+1} \xbar{{Q_i^p}_1}$, thus $\size I^p \leq \size I + c -1$.

Now suppose that $1 \in A_1,\ X_1^{a_1^1} \in G(Q_1),\ a_1^1 > 1,\ X_1,X_2^{a_2^u} \in G(Q_u)$ and $c_u = a_2^u > 1$. As we have seen above, we may assume that $c_1 = a_1^1$.
Then we have $\xbar{{Q_1^p}_1} \subset \sum_{j=2}^{w+1} \xbar{{Q_j^p}_1} + \xbar{{Q_u^p}_2}$, thus we get $\size I^p \leq \size I +c - 1$.
\end{proof}

\begin{example}
Consider the monomial ideal $I = (X_1^2,X_2) \cap (X_2,X_3) \cap (X_3,X_4) \cap (X_2,X_4) = \bigcap\limits_{i=1}^4 Q_i \subset \KK[X_1,X_2,X_3,X_4]$. Then $A = \{ \{1,2,3\},\{1,2,4\} \}$ and $c=1$, thus $\size I = 2$. Using Theorem \ref{th:main} we get that $\size I^p = \size  I + c$. Indeed, $\size I^p = \size (\x{1}{1},\x{2}{1}) \cap (\x{1}{2},\x{2}{1}) \cap (\x{2}{1},\x{3}{1}) \cap (\x{3}{1},\x{4}{1}) \cap (\x{2}{1},\x{4}{1})$.
\end{example}

\begin{example}
Let $I=(X_1^2,X_2) \cap (X_1,X_3) \cap (X_2,X_3) =  \bigcap\limits_{i=1}^3 Q_i \subset \KK[X_1,X_2,X_3]$. Then $A = \{ \{1,2\},\{1,3\}, \{2,3\} \}$ and $c=1$, thus $\size I =1$. We have that $\size I^p = \size (\x{1}{2},\x{2}{1}) \cap (\x{1}{1},\x{3}{1}) = \size \xbar{{Q_1^p}_2} \cap \xbar{Q_2^p} = 1$, thus $\size I^p < \size I + c$. We see that $I$ does not respect condition $(2) (A)$ from Theorem \ref{th:main}.
\end{example}

\begin{example}
Let $I = (X_1^2,X_2) \cap (X_3,X_4) \cap (X_1,X_4^2) = \bigcap\limits_{i=1}^4 Q_i \subset \KK[X_1,X_2,X_3,X_4]$. Then $A = \{ \{1,2\} \}$ and $c=2$ thus $\size I =1$. We have that $\size I^p = \size (\x{1}{2},\x{2}{1}) \cap (\x{3}{1},\x{4}{1}) \cap (\x{1}{1},\x{4}{2})= \size \bar{Q_1^p}_2 \cap \bar{Q_2^p} \cap \bar{Q_3^p}_2 = 2$, thus $\size I^p < \size I + c$. We see here that $X_4$ respects all the conditions from Theorem \ref{th:main}, but $X_1$ does not respect the condition $(2) (B)$.
\end{example}

\begin{example}

Let $I = (X_1^{k+1},X_2^k) \cap (X_1,X_2^{k+1}) \subset \KK[X_1,X_2].$ Then $\size I^p = \size I + c - k.$
Indeed, we have $\size I = 0, c = 2k$ and $\size I^p =\size (\x{1}{2},\x{2}{1}) \cap \ldots \cap (\x{1}{k+1},\x{2}{k}) \cap (\x{1}{1},\x{2}{k+1})$. We see that each variable appears only once, thus the size can not be smaller.
\end{example}

As an application of our main result, by using Theorem \ref{th:main} and Theorem \ref{teo:sdepth}, we easily deduce the following extension of Theorem \ref{teo:sdepth_size}.

\begin{corollary}\label{cor:sdepth} Let I be a monomial ideal of $S$ such that, either $I$ is squarefree, or $I$ is as described in Theorem \ref{th:main}. Then
\begin{enumerate}
\item $\sdepth I \ge \size I + 1;$
\item $\sdepth S/I \ge \size I.$
\end{enumerate}
\end{corollary}


\section{The behavior of size under generic deformations}\label{sec:deform}

The notion of \emph{deformation} of a monomial ideal was introduced by Bayer et al.~\cite{BPS} and further developed in Miller et al.~\cite{MSY}.

\begin{definition} (1) Let $\MM = \set{m_1, \dotsc, m_r} \subset S$ be a set of monomials. For $1\leq i\leq r$ let $a^i=(a^i_1, \dotsc, a^i_n) \in \NN^n$ denote the exponent vector of $m_i$.
	A \emph{deformation} of $\MM$ is a set of vectors $\ve_i = (\ve^i_1, \dotsc, \ve^i_n) \in \NN^n$ for $1 \leq i \leq r$
	subject to the following conditions:
	\begin{equation}\label{eq:deform}
	a^i_j > a^k_j \implies a^i_j + \ve^i_j > a^k_j + \ve^k_j \quad\text{ and }\quad a^i_k = 0 \implies \ve^i_k = 0.
	\end{equation}
	
	(2) Let $I \subset S$ be a monomial ideal with generating set $G_I$. A \emph{deformation} of $I$ is a deformation of  $G_I$.
	We set $\gen{I} := (g\cdot \xf^{\ve_g} \with g \in G_I)$  to be the ideal generated by the deformed generators.
\end{definition}

The most important deformations are the generic deformations. Let us recall the definition from \cite{MSY}.
\begin{definition} (1) A monomial $m \in S$ is said to \emph{strictly divide} another monomial $m' \in S$ if $m \mid \frac{m'}{X_i}$ for each variable $X_i$ dividing $m'$.

(2) A monomial ideal $I \subset S$ is called \emph{generic} if for any two minimal generators $m,m'$ of $I$ having the same degree in some variable, there exists a third minimal generator $m''$ that strictly divides $\lcm(m,m')$.

(3) A deformation of a monomial ideal $I$ is called \emph{generic} if the deformed ideal $\gen{I}$ is generic.
\end{definition}

\begin{definition} \cite{Mi}
A monomial ideals is considered generic if for any $m_1,m_2 \in$ Min $(I)$ and any variable $X_i$ if $deg_{X_i}(m_1) = deg_{X_i}(m_2),$ then $deg_{X_i}(m_1)=0.$ That is, no minimal generator of $I$ has the same exponent for any variable.
\end{definition}

\begin{proposition}\label{prop:def} Let $I \subset S=\KK[X_1,\ldots,X_n]$ be a monomial ideal and let $I=\bigcap\limits_{i=1}^rQ_i$ an irredundant decomposition
of $I$ as an intersection of irreducible ideals, where $\sqrt{Q_i} = P_i.$ Let $I_{\epsilon}$ be the generic deformation of $I.$ Then $\size I \geq \size I_{\epsilon}.$
\end{proposition}
\begin{proof}
We follow the proof of \cite{Mi} Theorem 6, except we consider any generic transformation, not just the generic ones. By construction, we have an unique map $\phi$ from the individual powers of variables amongst the generators of $I$ to individual powers of variables amongst the generators of $I_{\epsilon}.$ Let $I_{\epsilon} = \bigcap\limits_{i=1}^rQ_i$ be  the unique irredundant irreducible primary
decomposition of $I_{\epsilon}$, then by \cite{Mi} Theorem 6 we have that $I = \bigcap\limits_{i=1}^r \phi(Q_i)$ is an irreducible primary
decomposition of $I.$ The decomposition may not be irredundant.

If we consider $\size I_{\epsilon} = t < r$ with $\sqrt{\sum_{j=1}^tQ_{i_j}}=	\sqrt{\sum_{i=1}^rQ_i}$. then we also have
$$\sqrt{\sum_{j=1}^t \phi(Q_{i_j})}=	\sqrt{\sum_{i=1}^r \phi(Q_i)} = \sqrt{\sum_{i=1}^rQ_i}.$$
Thus we get the inequality $\size I \geq \size I_{\epsilon},$ because the decomposition for $I$ may be redundant.
\end{proof}

\begin{example}
Let $I = (xyt,xyw,xtw,yzt,yzw,ztw) \subset K[z,y,z,t,w]$ and let $\epsilon_1$ be a deformation and $\epsilon_2$ a generic deformation such that $I_{\epsilon_1} = (xyt,xyw,xt^3w,yzt^2,yzw,ztw)$ and $I_{\epsilon_2} = (x^3y^4t,x^2y^2w,xt^3w^3,y^3zt^2,yz^2w^2,z^3tw^4).$ Then we have the following unique irreducible primary decompositions:
\begin{flalign*}
I  = & (x,z) \cap (y,t) \cap (y,w) \cap (t,w); & \\
I_{\epsilon_1}  = & (x,z) \cap (y,t) \cap (y,w) \cap (t,w) \cap (x,t^2,w) \cap (y,z,t^3);\\
I_{\epsilon_2}  = & (x,z) \cap (y,t) \cap (y^3,y^2w,yw^2,w^3) \cap (t,w) \cap (y^4,y^3t^2,y^2w,yw^2,t^3,w^4) \cap \\
                  & (x^3,x^2w,t^2,w^2) \cap (y^4,y^3z,y^2w,z^2,w^3) \cap (x^3,x^2w,z,w^3) \cap (x^2,xw^3,y^3,yw^2,w^4) \cap \\
                  & (y^2,yz^2,z^3,t^3) \cap (x^2,z^2,zt^2,t^3) \cap (x,y^3,yz^2,z^3) \cap  (x^3,x^2y^2w,x^2z^2,y^3,yz^2w^2,z^3,w^3) \cap \\
                  & (x^3,x^2y^2,y^3,yz^2,z^3,t^3).
\end{flalign*}
We have that $2 = \size I > \size I_{\epsilon_1} = \size I_{\epsilon_2} = 1$, thus the inequality from the above proposition may be strict.
\end{example}

\section*{Acknowledgements}


They were partially supported  by project  PN-II-RU-TE-2012-3-0161, granted by the Romanian National Authority for Scientific Research,
CNCS - UEFISCDI, during the preparation of this work.

\end{document}